\documentclass{article}
\usepackage{latexsym,amsmath,amsthm,amssymb,cite,graphicx}
\numberwithin{equation}{section}
\newtheorem{theorem}{Theorem}[section]
\newtheorem{lemma}[theorem]{Lemma}

\theoremstyle{definition}
\newtheorem{definition}[theorem]{Definition}
\newtheorem{example}[theorem]{Example}
\theoremstyle{remark}

\begin{document}
\title{Fractional Differential Equations with  Periodic Boundary Conditions of Constant Ratio}
\author{{Anwarrud Din}\\
\small\textit{Department of Mathematics, Sun Yat-sen University, Guangzhou 510275, China}\\
\small\textit{{E}-mail: anwarm.phil@yahoo.com}\\
{Shah Faisal}\\
\small\textit{Department of Mathematics, University of Peshawar, Peshawar 25000, Pakistan}\\
\small\textit{{E}-mail: shahfaisal8763@gmail.com}}
\date{}
\maketitle
\begin{abstract}
This article  is concerned with the existence and uniqueness of solutions to some  fractional order boundary value problems
of the type
 \begin{equation*}\begin{split}
^{c}D^\alpha_{0+} u(t)= f(t, u(t),^{c}D^{\alpha-1}u(t)) , ~~~~~
1<\alpha\leq 2,\, t\in J=[0,1]\\
u(0)= \xi u(1),~~ ^{c}D^\beta u(0)= \xi \,\,^{c}D^{\beta}u(1),~~
0<\beta<1,~~\xi\in ,(0,1), \end{split}\end{equation*} where $
^{c}D^{\alpha}_{0+}$ represents  Caputo fractional order derivative
and $f:J\times \mathbb{R}\times \mathbb{R}\rightarrow \mathbb{R}$ is continuous
function explicitly depending on fractional order derivative. Our
results are based on  some fixed point theorems of functional
analysis. For the applicability of our results, we provide an
example.
\end{abstract}
\section{Introduction}
In the last few decades it was realized that in some real world problems, fractional order models are more adequate and accurate
compared to their counterpart integer-order models. The main advantage of fractional models in comparison with classical integer
models is that it provide an excellent tool for description of memory and heredity properties of various  materials and processes. This theory has many application in different field of sciences like engineering, physics, chemistry, biology etc,  see for example
\cite{4,5,6,7,8,9,10} and reference therein. In consequence, the subject of fractional order differential equations gaining much
importance and attentions. But it should be noted that most of the available literature  are devoted to the solvability of linear
initial value problems for fractional order differential equation in term of special functions \cite{3,4,5,6}.

Existence theory corresponding to boundary value problems for fractional order differential equations had attracted the attention
of researchers quite recently. There are some work dealing with the existence and multiplicity of positive solution to non-linear
initial value problems associated with fractional order differential equation, \cite{7,8,9,10}. Recently, M. Benchohra and N. Hamidi and J. Henderson
\cite{11} studied non-linear fractional order differential equation with periodic boundary. However, the problem with the boundary
conditions we study had never been studied before. We are concern with the  existence and uniqueness of positive solutions to boundary
value problem of the form
\begin{equation}\label{eqn1}\begin{split}
^{c}D^\alpha_{0+} u(t)= f(t, u(t),^{c}D^{\alpha-1}u(t)) , ~~~~~
1<\alpha\leq 2,\, t\in J=[0,1]\\u(0)= \xi u(1),~~ ^{c}D^\beta u(0)=
\xi \,\,^{c}D^{\beta}u(1),~~  0<\beta<1,~~\xi\in ,(0,1).
\end{split}\end{equation} We use results from the functional
analysis.
We recall some notations, definition and necessary lemmas  \cite{1,2,3} necessary for our investigation.
  The Banach space of all continuous functions from $J=[0, 1]$ into
$\mathbb{R}$ with the norm $ \|y\|_{\infty}=\sup\{|y(t)|; 0\leq t\leq 1\} $
is denoted by $ C(J,\mathbb{R})$ and   Banach space of Lebesgue integrable
function with the norm $\|y\|_{L^{1}}=\int_{0}^{1}|y(t)|dt$ is
denoted by  $L^{1}(J,\mathbb{R})$. Let us denote by $\tilde{C}(J,\mathbf{R})=\{u\in
C(J,\mathbb{R}),\,^{c}D^{\alpha-1}u\in C(J,\mathbb{R})\}$, then  $\tilde{C}(J,\mathbb{R})$
is a Banach space under the norm
$||u||_{\tilde{c}}=\max\{||u||_{\infty},||^{c}D^{\alpha-1}u||_{\infty}\}$.
\begin{definition}
The fractional order integral of the function $h\in l^1(J,\mathbb{R})$ of order $ \alpha\in \mathbb{R}$, is defined by
\begin{equation*}
I^{\alpha}_{0}f(t)=\frac{1}{\Gamma\alpha}\int ^{t}_{0}(t-s)^{\alpha-1}f(s)ds,
\end{equation*}
where $ \Gamma $ is the gamma function.
\end{definition}
\begin{definition}\label{d2}
For a  function $f$ the $\alpha$th order Caputo fractional
derivative is define  by
\begin{equation*}
(^{c}D^{\alpha}_{0+})f(t)=\frac{1}{\Gamma(n-\alpha)}\int ^{t}_{0}
(t-s)^{n-\alpha-1}f^{(n)}(s)ds,
\end{equation*}
where $n$ is integer such that $ n=\lceil \alpha\rceil$.
\end{definition}
\begin{lemma}
Let $\alpha>0$ and $u\in C(0,1)\cap L(0,1)$,  the fractional order differential equation
\begin{equation*}
D^{\alpha}_{0+}u(t)=0,\,n-1<\alpha<n
\end{equation*}
has a unique solution of the form
\begin{equation*}
u(t)=C_{0}+C_{1}t+........+C_{n-1}t^{n-1},~~~~ C_{i}\in \mathbb{R},~~~~
i=1,2,3........,n-1.
\end{equation*}
\end{lemma}
\begin{lemma}\label{l4}
Let $\alpha> 0$  then
\begin{equation*}
I^{\alpha}\,^{c}D_0^{\alpha}u(t)=u(t)+C_{0}+C_{1}t+........+C_{n-1}t^{n-1},~~~~\mbox{for},\,
C_{i}\in \mathbb{R},~~~~ i=1,2,3........,n-1,\, \alpha\leq n< \alpha+1
\end{equation*}
\end{lemma}
\section{Main results}
Now we study sufficient conditions for existence and uniqueness of
solutions.
\begin{lemma}\label{a} For $y\in C(J)$, the linear fractional order boundary value problem
\begin{equation*}\begin{split}
^{c}D^\alpha u(t)= y(t), ~~~~~ 1<\alpha\leq 2,\,\, t\in J=[0,1]\\
u(0)= \xi u(1),\, ^{c}D^\beta u(0)= \xi\,\,^{c}D^{\beta}u(1), \,
0<\beta<1~~\xi\in (0,1)
\end{split}\end{equation*} has a solution of the form $ u(t)= \int_{0}^{1} G(t, s)y(s)ds,$
where the Green function is
\begin{eqnarray*}
G(t,s)&=&\left\{  \begin{array}{l}
\frac{(1-\xi)(t-s)^{\alpha-1}+\xi(1-s)^{\alpha-1}}{\Gamma\alpha(1-\xi)}-\frac{\Gamma(2-\beta)(\xi+(1-\xi)t)}{\Gamma(\alpha-\beta)(1-\xi)}(1-s)^{\alpha-\beta-1}, \\ \hspace{7.5 cm} \,\mbox{if}\,\,\,\,\,~0\leq s\leq t\leq 1\cr\cr
\frac{{\xi(1-s)^{\alpha-1}}}{{\Gamma\alpha(1-\xi)}}-\frac{\Gamma(2-\beta)(\xi+(1-\xi)t)}{\Gamma(\alpha-\beta)(1-\xi)}(1-s)^{\alpha-\beta-1},\hspace{1 cm} \mbox{if}\,\,\,~0\leq t\leq s\leq 1,
\end{array}\right.
\end{eqnarray*}
\end{lemma}
\begin{proof}
By lemma (\ref{l4}), solution of the fractional order differential
equation $^{c}D^\alpha u(t)= y(t)$ is given by
\begin{equation*}
 u(t)=I^{\alpha} y(t)+C_{0} +C_{1}t\text { and }^{c}D^\beta u(t)= I^{\alpha-\beta} y(t)+C_1 \frac{t^{1-\beta}}{\Gamma (2-\beta)}.
\end{equation*}
Using the boundary conditions $u(0)=\xi u(1)$ and $^{c}D^\beta u(0)=
\xi ^{c}\,\,D^\beta u(1)$, we obtain
\begin{equation*}
C_{0}= \frac{\xi}{1-\xi} I^\alpha y(1)-\Gamma
(2-\beta)\frac{\xi}{1-\xi}I^{\alpha- \beta}y(1), \text{ and }C_{1}=
-\Gamma (2-\beta) I^{\alpha- \beta}.
\end{equation*}
Hence $u(t)=I^\alpha y(t)+\frac{\xi}{1-\xi}I^\alpha y(1)-\Gamma
(2-\beta)(\frac{\xi}{1-\xi}+t)I^{\alpha-\beta} y(1)$ which implies
that
\begin{eqnarray*}\begin{split}
u(t)&= \frac{1}{\Gamma \alpha} \int_0^t(t-s)^{\alpha-1} y(s)ds+
\frac{\xi}{1-\xi} \int_0^1(1-s)^{\alpha-1}y(s)ds\\&
-\frac{\Gamma(2-\beta)}{\Gamma(\alpha-\beta)}\big ( t+
\frac{\xi}{1-\xi})\int_0^1(1-s)^{\alpha-\beta-1}y(s)ds=\int_0^1 G(t,
s)y(s)ds.
\end{split}\end{eqnarray*}
\end{proof}
 If $ \alpha-\beta<1 $, the Green function $ G(t, s)$ become
unbounded but the function $ t:\rightarrow  \int_0^1 G(t, s)y(s)ds $
is continuous on $J$ so attain its spermium value say $ G^*
=\sup_{t\in j} \int_0^1 \mid G(t,s)\mid ds$. In view of Lemma
\eqref{a}, an equivalent representation of the BVP \eqref{eqn1} is
given by
\begin{equation} u(t)= \int_0^1 G(t,
s)f(s, u(s),^{c}D^{\alpha-1}u(s))ds,\,t\in J.
\end{equation}$T:\tilde{C}(J,\mathbb{R})\rightarrow
\tilde{C}(J,\mathbb{R})$ by
\begin{equation}\label{a3} Tu(t)=\int_0^1 G(t,s)f(s,u(s), ^{c}D^{q-1}u(s))\,ds,
\end{equation} then solutions of the BVP \eqref{eqn1}  are fixed points
of $T$.
\begin{theorem}
Assume that $f: J\times \mathbb{R} \times \mathbb{R}\rightarrow \mathbb{R}$ is continuous and the following hold\\
$(A_1)$ there exist $p \in C(J, \mathbb{R}^+)$ and a continuous,
non-decreasing function $\psi: [0, \infty] \rightarrow(0, \infty)$
such that
\begin{equation*}
\mid f(t, u, z)\mid \leq p(t) \psi(\mid z \mid)
~~~~~~~~~~~~~~~\mbox{for all}\, \,t\in J,\, u,\, z \in \mathbb{R}.
\end{equation*}
$(A_2)$ there exist constant $r > 0$ such that
\begin{equation*}
r\geq \max\big\{G^* p^*\psi(r), \,\,p^* \psi(r)
\frac{\Gamma(3-\alpha)\Gamma(\alpha-\beta+1)+\Gamma(2-\beta)}{\Gamma(3-\alpha)\Gamma(\alpha-\beta+1)}\big\},
\end{equation*}
where $p^*= \sup\big\{p(s),s\in J\big\},$

then the BVP \eqref{eqn1} has at lest one solution such that
$|u(t)|<r $  on $J$.
\end{theorem}

\begin{proof}
We prove the result via Schauder fixed point theorem.
Firstly, we prove that the operator $T$ defined by \eqref{a3} is continuous. Choose $r$ as in $(A_2)$ and define$
D=\{u\in\tilde{c}(J,\mathbb{R}),\,||u||_{\tilde{c}}\leq r\}$ a closed and
bounded subset of $\tilde{C}(J,\mathbb{R})$. Let the sequence $\{u_{n}\}$
converges to $u$ in $\tilde{C}(J,\mathbb{R})$, then as in the proof of Lemma
$2,5$ of \cite{11}, $D^{\alpha-1}u_{n}\rightarrow D^{\alpha-1}u$.
Choose $\rho
>0$ such that
\begin{equation*}
 ||u_{n}||_{\tilde{C}}\leq\rho, \,\, ||u||_{\tilde{C}}\leq \rho.
\end{equation*}
For for all $t\in J$, we have
\begin{equation*} |Tu_{n}(t)-Tu(t)|
\leq
\int^{1}_{0}|G(t,s)[f(s,u_{n}(s),D^{\alpha-1}u_{n}(s))-f(s,u(s),D^{\alpha-1}u(s))]|ds\text{
and}
\end{equation*}
\begin{eqnarray*}\begin{split}
&|D^{\alpha-1}Tu_n(t)-D^{\alpha-1}Tu(t)| \leq
\int_0^t|f(s,u_n(s),D^{\alpha-1}u_n(s))-f(s,u(s),D^{\alpha-1}u(s))|ds
\\&
+\frac{\Gamma(2-\beta)}{\Gamma(3-\alpha)(\Gamma(\alpha-\beta)}t^{2-\alpha}\int_0^1(1-s)^{\alpha-\beta-1}|f(s,u_n(s),D^{\alpha-1}u_n(s))
-f(s,u(s),D^{\alpha-1}u(s))|ds .\end{split}\end{eqnarray*} From the
continuity of $f(s,u(s),D^{\alpha-1}u(s))$ and Lebesgue dominated
convergence theorem, it follows that
  \begin{equation*}\begin{split}
&||Tu_{n}(t)-Tu(t)||_\infty \rightarrow 0\text{ and }
||D^{\alpha-1}Tu_n(t)-D^{\alpha-1}Tu(t)||_\infty\rightarrow 0
\end{split}\end{equation*} as $n\rightarrow\infty$ which implies that $T$ is continuous.

 Now we show that  $TD\subseteq D$.
Let $ u\in D$ then for each $t\in J$ ,we have
\begin{eqnarray*}
|Tu(t)|&\leq&
\int^{1}_{0}(|G(t,s)||f(s,u(s),^{c}D^{\alpha-1}u(s))|)ds \leq
G^{\star}\int^{1}_{0}(|f(s,u(s),^{c}D^{\alpha-1}u(s))|)ds\\
&\leq& G^{\star}p^{\star}\psi(||u||_{\tilde{c}})\leq
G^{\star}p^{\star}\psi(r),
\end{eqnarray*}

 \begin{eqnarray*}
 |^{c}D^{\alpha-1}Tu(t)|&\leq& \int_0^t|f(s, u(s), ^{c}D^{\alpha-1}u(s))|ds \\&&+\frac{\Gamma(2-\beta)}{\Gamma(3-\alpha)(\Gamma(\alpha-\beta))}t^{2-\alpha}
 \int_0^1(1-s)^{\alpha-\beta-1}|f(s, u(s), ^{c}D^{\alpha-1}u(s))|ds
 \end{eqnarray*}
 \begin{equation*}
 ||^{c}D^{\alpha-1}Tu(t)||_\infty\leq P^*\psi(\|u\|_{\tilde{c}})(\frac{\Gamma(3-\alpha)\Gamma(\alpha-\beta+1)+
 \Gamma(2-\beta)}{\Gamma(3-\alpha)\Gamma(\alpha-\beta+1)}).
 \end{equation*}
 Consequently,
\begin{eqnarray*}
 \|Tu(t)\|_{\tilde{c}}&\leq& \max\bigg\{{G^{\star}p^{\star}\psi(r),
 \,\, p^{\star}\psi(r)(\frac{\Gamma(3-\alpha)\Gamma(\alpha-\beta+1)+\Gamma(2-\beta)}{\Gamma(2-\beta+1)\Gamma(3-\alpha)})}\bigg\}
\leq r
\end{eqnarray*}
implies that $Tu(t)\in  D$ for all $u(t)\in D$.

Finally, we show that $T$ maps $D$ into equicontinuous set of
$\tilde{C}(J,\mathbb{R})$. Take $t_{1},t_{2}\in J$ such that $t_{1}<t_{2}$
and $u\in D$, we have
\begin{eqnarray*}
|Tu(t_{2})-Tu(t_{1})|&\leq&\int^{1}_{0}|G(t_{2},s)-G(t_{1},s)||f(s,u(s),^{c}D^{\alpha-1}u(s))|ds\\
&\leq&
p^{\star}\psi(||u||_{\tilde{c}})\int^{1}_{0}|G(t_{2},s)-G(t_{1},s)|ds,
\end{eqnarray*}
In view of the continuity of $G$, it follows that
\begin{equation*}
\|Tu(t_{2})-Tu(t_{1})\|\rightarrow 0 \text{ as }t_2\rightarrow t_1.
\end{equation*}
Further,
\begin{eqnarray*}
|(T ^{c}D^{\alpha-1}u)(t_{2})-(T ^{c}D^{\alpha-1}u)(t_{1})|&\leq&
\bigg|\int^{t_2}_0
f(s,u(s),^{c}D^{\alpha-1}u(s)ds)-\int^{t_1}_{0}f(s,u(s),^{c}D^{\alpha-1}u(s))ds\bigg|\\&&+\frac{(t_1^{2-\alpha}
-t_2^{2-\alpha})\Gamma(2-\beta)}{\Gamma(3-\alpha)\Gamma(\alpha-\beta)}|\int^{1}_{0}(1-s)^{\alpha-\beta-1}f(s,u(s),^{c}D^{\alpha-1}u(s))ds|\\
&\leq&p^{*}\psi(\max(\|u\|_\infty,
\|^{c}D^{\alpha-1}u(t))\|_{\infty})\bigg\{t_{2}-t_{1}+
\frac{(t_1^{2-\alpha}-t_2^{2-\alpha})\Gamma(2-\beta)}{\Gamma(3-\alpha)\Gamma(\alpha-\beta+1)}\bigg\}\\
&\leq&
p^{*}\psi(\|u\|_{\tilde{c}})\bigg\{t_{2}-t_{1}+\frac{(t_1^{2-\alpha}-t_2^{2-\alpha})\Gamma(2-\beta)}{\Gamma(3-\alpha)\Gamma(\alpha-\beta+1)}\bigg\},
\end{eqnarray*}
which implies that
\begin{equation*}
|T ^{c}D^{\alpha-1}u)t_{2}-(T ^{c}D^{\alpha-1}u)t_{1}|\rightarrow 0\
\text{ as }t_{2}\rightarrow t_{1}.
\end{equation*}
Hence
\begin{equation*}
\|(Tu)t_{2}-(Tu)t_{1}\|_{\tilde{C}}=\max\bigg\{\|(Tu)t_{2}-(Tu)t_{1}\|_\infty,
\|T({c}D^{\alpha-1}u)t_{2}-(T{c}D^{\alpha-1}u)t_{1}\|_\infty\bigg\}
\rightarrow 0\text{ as }t_{2}\rightarrow t_{1}.
\end{equation*}

By Arzela-Ascoli Theorem, $T$ is completely continuous  and by
Schauder fixed point theorem, $T$ has a fixed point $u$ in $D$
 such that $|u(t)|< r $ for all $t\in J$.\end{proof}
\begin{theorem}
Assume that $f:J\times \mathbb{R}\times \mathbb{R}\longrightarrow \mathbb{R}$ is continuous
and there exist a constant $k > 0$ such that
\begin{equation*}
\mid f(t,u,v)-f(t,\bar{u},\bar{v})\mid \leq k(\mid u-\bar{u}\mid
+\mid v-\bar{v}\mid),\text{ for }t\in J \text{ and } u, v, \bar{u},
\bar{v} \in \mathbb{R}.
\end{equation*}
If
\begin{equation*}
 max \{ 2G^*k , \,2k\frac{(\Gamma3-\alpha)(\Gamma\alpha-\beta +1)+\Gamma (2-\beta)}{\Gamma(3-\alpha)(\Gamma(\alpha-\beta+1)}\}<1
\end{equation*}
then the BVP (\ref{eqn1}) has a unique solution.
\end{theorem}
\begin{proof} The proof is based on Banach Fixed point theorem, we show that the operator
$T$ is contraction. For $u, \bar{u} \in \tilde{C}(J,\mathbb{R})$ and  $t \in
J$, we have
\begin{eqnarray*}
\mid (Tu)t- (T\bar{u})t\mid &\leq& sup_{t\in j}   \int_0^1 G(t,s)\mid f(s, u(s), ^{c}D^{\alpha-1}u(s))-f(s, \bar{u}(s), ^{c}D^{\alpha-1}\bar{u}(s))\mid ds\\
&\leq&  G^* k \{\mid u- \bar{u}\mid+\mid D^{\alpha-1}u
-D^{\alpha-1}\bar{u}\mid \}\leq 2 G^* k \parallel u-
\bar{u}\parallel_{\tilde{c}}
\end{eqnarray*}
and
\begin{eqnarray*}\begin{split}
&\mid D^{\alpha-1}Tu(t) -D^{\alpha-1}T\bar{u}(t)\mid \leq \int_0^t
\mid f(s, u(s), D^{\alpha-1}u(s))-f(s, \bar{u}(s),
D^{\alpha-1}\bar{u}(s))\mid ds\\&+
\frac{\Gamma(2-\beta)t^{2-\alpha}}{\Gamma(3-\beta)\Gamma(\alpha-\beta)}\int_0^1
(1-s)^{\alpha-\beta-1}
\mid f(s, u(s), D^{\alpha-1}u(s)-  f(s, \bar{u}(s), D^{\alpha-1}\bar{u}(s))\mid ds\\
&\leq
(2k+\frac{2k\Gamma(2-\beta)}{\Gamma(3-\alpha)\Gamma(\alpha-\beta+1)})||u-\tilde{u}||_{\tilde{c}}
\leq
2k(\frac{\Gamma(3-\alpha)\Gamma(\alpha-\beta+1)+\Gamma(2-\beta)}{\Gamma(3-\alpha)\Gamma(\alpha-\beta+1)})||u-\tilde{u}||_{\tilde{c}}.
\end{split}
\end{eqnarray*}
Since
\begin{equation*}
\max\bigg\{2G^* k, 2k \bigg(\frac{\Gamma(3-\alpha)\Gamma(\alpha-\beta+1)+\Gamma(2-\beta)}{\Gamma(3-\alpha)\Gamma(\alpha-\beta+1)}\bigg)\bigg\}=d<1
\end{equation*}
it follows that
\begin{equation*}
\parallel Tu(t)-T\bar{u}(t) \parallel_c \leq d\parallel
u-\bar{u}\parallel_{\tilde{c}}
\end{equation*}
 and by Banach fixed point theorem, the B.V.P (\ref{eqn1}) has a unique
solution.
\end{proof}

\begin{example}
Consider the fractional boundary value problem
\begin{equation*}
^{c}D^{{\frac {3}{2}}}u(t)=\frac{\sin^{2}t}{11(e^{2t}+3e^{t}+1)}\bigg(3+t+5u(t)+D^{\frac{1}{2}}u(t)\bigg),
\end{equation*}
where
$u(0)=\frac{1}{2}u(1)$,\,\, $^{c}D^{\frac{1}{2}}u(0)=\frac{1}{2},\,\,^{c}D^{\frac{1}{2}}u(1)$ for $\xi=\beta=\frac{1}{2}$.\\
Here
\begin{equation*}
 f(t,u(t),D^{\frac{1}{2}}u(t))=\frac{\sin^{2}t}{11(e^{2t}+3e^{t}+1)}\bigg(3+t+5u(t)+D^{\frac{1}{2}}u(t)\bigg),
 \end{equation*}
 therefore
\begin{eqnarray*}
|f(t, u(t),D^{\frac{1}{2}}u(t))-f(t, \bar{u}(t),D^{\frac{1}{2}}\bar{u}(t))|&\leq& \frac{1}{11}(|u-\bar{u}|+|D^{\frac{1}{2}}u-D^{\frac{1}{2}}\bar{u}|)\\
&\leq& k(|u-\bar{u}|+|v-\bar{v}|)
\end{eqnarray*}
which is the condition $(A_{1})$ with $k=\frac{1}{11}.$\\
Also\\
\begin{eqnarray*}
G(t,s)&=&\left\{  \begin{array}{l}
\frac{(1-\xi)(t-s)^{\alpha-1}+\xi(1-s)^{\alpha-1}}{\Gamma\alpha(1-\xi)}-\frac{\Gamma(2-\beta)(\xi+(1-\xi)t)}{\Gamma(\alpha-\beta)(1-\xi)}(1-s)^{\alpha-\beta-1}, \\ \hspace{8 cm} \,\mbox{if}\,\,\,\,\,~0\leq s\leq t\leq 1\cr\cr
\frac{{\xi(1-s)^{\alpha-1}}}{{\Gamma\alpha(1-\xi)}}-\frac{\Gamma(2-\beta)(\xi+(1-\xi)t)}{\Gamma(\alpha-\beta)(1-\xi)}(1-s)^{\alpha-\beta-1},\hspace{.5 cm} \mbox{if}\,\,\,~0\leq t\leq s\leq 1,
\end{array}\right.
\end{eqnarray*}
Or
\begin{equation*}
G(t,s)=\bigg\{\begin{array}{c}
      \frac{\frac{1}{2}(t-s)^{\frac{1}{2}}+\frac{1}{2}(1-s)^{\frac{1}{2}}}{\frac{1}{2}\Gamma(\frac{3}{2})}- \frac{\Gamma(\frac{3}{2})(\frac{1}{2}+\frac{1}{2}t)}{\frac{1}{2}\Gamma(1)}\\
       \frac{\frac{1}{2}(1-s)^{\frac{1}{2}}}{\frac{1}{2}\Gamma(\frac{3}{2})}- \frac{(\frac{1}{2}+\frac{1}{2}t)\Gamma(\frac{3}{2})}{\frac{1}{2}\Gamma(1)}
    \end{array}
\end{equation*}
so that
\begin{eqnarray*}
G^{*}\leq \frac{\int^{1}_{0}(1-s)^{\frac{1}{2}}ds}{\frac{1}{2}\Gamma(\frac{3}{2})}+ \frac{\Gamma(\frac{3}{2})}{\frac{1}{2}}
\leq \ 3.1601\cdots.
\end{eqnarray*}which implies that
\begin{equation*}
 2kG^{*}\leq (2)(\frac{1}{11})(3.1601)=0.5745...<1.
 \end{equation*}
And
\begin{eqnarray*}
 2k\big(\frac{\Gamma(3-\alpha)\Gamma(\alpha-\beta+1) + \Gamma(2-\beta)}{\Gamma(3-\alpha)(\Gamma\alpha-\beta+1)}\big)
 = \frac{2}{11}(\frac{\Gamma(\frac{3}{2})\Gamma(2) +\Gamma(\frac{3}{2})}{\Gamma(\frac{3}{2}) \Gamma(2)}) = 0.3636\cdots< 1.
 \end{eqnarray*}
 Or
\begin{equation*}
 \max\{2kG^{*} , \frac{2k(\Gamma(3-\alpha)\Gamma(\alpha-\beta+1) +\Gamma(2-\beta)}{\Gamma(3-\alpha)\Gamma(\alpha-\beta+1)}\}< 1.
\end{equation*}
Hence by Banach  contraction mapping the given fractional bounded
value problem has a unique solution on $J \in [0,1].$
\end{example}
{\bf Conflict of interests}.  The authors declare that they have no
competing interests.
\newpage


\end{document}